\documentclass[12pt]{amsart}

\usepackage{amsfonts}
\usepackage{amssymb}
\usepackage{amsmath} 
\usepackage{amsthm}
\usepackage{enumerate}
\usepackage[all]{xy}
\usepackage{graphics}
\usepackage{graphicx}  
\usepackage{appendix}
\usepackage{caption}
\usepackage{pdfpages}
\usepackage[bottom]{footmisc}
\normalsize

\newcommand{\te}{Teich\-m\"ul\-ler}
\newcommand{\tes}{Teich\-m\"ul\-ler's}

\newcommand{\D}{\Delta}

\theoremstyle{remark}

\newcommand{\Car}{Carath\'eodory}

\newtheorem{theorem}{Theorem}

\theoremstyle{definition}
\newtheorem{definition}{Definition}

\theoremstyle{definition}

\newtheorem{lemma}{Lemma}

\begin{document}

\title{Carath\'eodory's and Kobayashi's metrics on Teichm\"uller space}

\author{Frederick P. Gardiner}
\address{Department of
Mathematics, Graduate 
Center of CUNY, NY 10016.}
\curraddr{}
%email{frederick.gardiner@gmail.com}

%\date{June 2017}

\thanks{
      I am extremely grateful to William Harvey for his help in the preparation of this manuscript.  Also to Nikola Lakic who is primarily responsible for developing the theory of the \te\ density which provides the bridge for comparing the lengths of certain curves in \te\ space to corresponding curves on Riemann surfaces (see appendix D). Also thanks are due to Clifford Earle, Irwin Kra and Stergios Antonakoudis for helpful conversations and comments.}


%\begin{center}
%\bf
%  \renewcommand{\thefootnote}{}
%  \footnote{2010 {\it Mathematics Subject Classification}.
%    Primary 30F60; Secondary 32G15, 30C70, 30C75.}
  %
  %   Restore normal footnote numbering (not necessary if there are no
  %    subsequent numbered footnotes).
 % \renewcommand{\thefootnote}{\arabic{footnote}}%
%\end{center}

\begin{center}

version6, April 16th, 2019

\end{center}

\begin{abstract}  Carath\'eodory's and  Kobayashi's infinitesimal metrics on \te\ spaces of dimension two or more  are never equal in the direction of any tangent vector defined by a separating cylindrical differential. 
\end{abstract}

\maketitle

\begin{section}*{Introduction}  
This paper deals with infinitesimal and global conformal metrics on a complex manifold $M.$ In the case at hand M will be either a \te\ space or a Riemann surface.  We use the notation $d(\cdot,\cdot)$  for both global and infinitesimal metrics. If the first entry in the parentheses is a point  $p$ in $M$ and the second  is a tangent vector $V$ at $p,$ then $d(p,V)$ will be an infinitesimal form,  namely, the assignment of a  norm $\|V\|_p$  for the tangent vector  $V$ at $p.$  It will depend in a Lipschitz manner on $p.$  Metrics of this type are called Finsler metrics.  If both entries of $d(\cdot,\cdot)$ are points $p_1$ and $p_2$ in $M$ then the meaning changes;  
$d(p_1,p_2)$ means the integrated form of $d(p,V).$  By this we mean 
$d(p_1, p_2)$ is the infimum of the arclength integrals
\begin{equation}\label{def100}
\int_{\gamma} d(\gamma(t), \gamma'(t)) dt,
\end{equation}
where the infimum is taken over all piecewise differentiable curves $\gamma(t)$ that join $p_1$ to $p_2$ in $M.$

The most important example is when $M$ is conformal to the open unit disc in the complex plane ${\mathbb C},$ that is, 
$$M \cong \Delta= \{z \in {\mathbb C}: |z|<1\}$$ and the metric is the Poincar\'e metric, which we denote by $\rho.$ 
With this convention the infinitesimal form is 
\begin{equation}\label{200}\rho(p,V)=\frac{|dp(V)|}{1-|p|^2},\end{equation}
where $V$ is a tangent vector at a point $p$ with $|p|<1.$  The 
global form of $\rho$ when the disc is $\Delta$ is 
\begin{equation}\label{one}\rho(z_1,z_2)=\frac{1}{2} \log \frac{1+r}{1-r},
\end{equation}
where $r=\frac{|z_1-z_2|}{|1-\overline{z_1}z_2|}.$

The infinitesimal form $\rho$  induces  infinitesimal forms on any 
complex manifold $M$ in two ways.  The first uses  the family $\mathcal{F}$ of all holomorphic functions  $f$ from $\Delta$ to $M.$  It is commonly called Kobayashi's metric and we denote it by 
$K.$ The second uses  the family $\mathcal{G}$ of all holomorphic functions $g$ from $M$ to $\Delta.$  It is commonly called  \Car's metric and we denote it by $C.$    

    In general if $f:M \rightarrow N$ is a mapping from a manifold $M$ to a manifold $N$ and $f(p)=q,$
    we denote by $(df)_p$ the induced linear map from the fiber of  tangent space to $M$ at $p$ to the fiber of the tangent space to $N$ at $q.$ 

For a point $p$ in a complex manifold $M$ we let  
${\mathcal F}({p})$ be the subset of $f \in {\mathcal F}$ for which $f(0)={p}$ and  
 ${\mathcal G}(p)$ consist of those functions $g \in {\mathcal G}$ for which $g({p})=0.$  Given a point $p \in M$ and a vector $V$ tangent to $M$ at $p,$ the infinitesimal form $K$ is  defined by the infimum  

 \begin{equation}\label{K}
 K_M(p,V) = \inf_{ f \in {\mathcal F(p)}} \left\{ \frac{1}{|a|}:  df_0(1)=a V {\rm \ and \ } f \in {\mathcal F}(p)\right\}  
 \end{equation}
and the infinitesimal form $C$ is defined by the supremum 
 \begin{equation}\label{C111}
 C_M(p,V) = \sup_{ g \in {\mathcal G(p)}} \left\{|b|:  dg_p(V)=b {\rm \ and \ } g \in {\mathcal G}(p)\right\}.
 \end{equation}
 When $M=\Delta$  Schwarz's lemma shows that the metrics $K_{\Delta},$ $C_{\Delta}$ and  $\rho$ coincide.

   For general manifolds $M,$ Schwarz's lemma applied to $g \circ f : \Delta \rightarrow \Delta$ where  $f \in {\mathcal F}(\tau)$ and $g \in {\mathcal G}(\tau)$ implies 
 \begin{equation}\label{S111}|b| \leq 1/|a| 
 \end{equation}
 and so the formulas (\ref{K}), (\ref{C111}) and (\ref{S111}) together imply  for all 
 complex manifolds $M$
 \begin{equation}\label{C<K}
 C_M(p,V) \leq K_M(p,V).
 \end{equation}
 
 The two metrics $C$ and $K$ represent categorical extremes of a general property of so called Schwarz-Pick metrics.  $C$ is the smallest conformal metric for which holomorphic mappings are contracting and $K$ is the largest.  This idea is explained further in Appendix A.

   In the case $M$ is  a \te\ space $Teich(R)$ one can use the Bers embedding  \cite{Ahlforsbook5, Bers11} to obtain a reverse inequality and ends up with the double inquality
    \begin{equation}\label{equivCK}
 (1/3) K_{Teich(R)}(\tau,V) \leq C_{Teich(R)}(\tau,V) \leq K_{Teich(R)}(\tau,V).
 \end{equation}
 Since the arbitrary points are now points in a \te\ space, we have switched notations and have denoted those points  by the letter  $\tau$ instead of $p.$  We refer to Appendix B for the proof of the left hand inequality in  (\ref{equivCK}).

\bigskip

  A global metric $d$  satisfying mild smoothness conditions  has an infinitesimal form  given by the limit:
 $$ d(\tau,V) = \lim_{t \searrow 0} d(\tau,\tau+tV).$$
 (See \cite{EarleHarrisHubbardMitra}.)
  When this is the case the integral of the infinitesimal form $d(\tau, V)$  gives back a global metric 
  $\overline{d}.$   By definition the metric $\overline{d}(\tau_1,\tau_2)$ on pairs of points $\tau_1$ and $\tau_2$  is the infimum of the arc lengths of arcs in the manifold that join 
  $\tau_1$ to $\tau_2.$    In general  
  $\overline{d}$ is symmetric, satisfies the triangle inequality, determines the same topology as $d$ and one always has the inequality
  $\overline{d} \geq d.$  In many situations $\overline{d}$ is larger than $d,$ 
  but, when $M$ is a \te\ space and $d$ is \tes\ metric $d$ and $\overline{d}$ coincide, \cite{EarleEells} \cite{OByrne}.  Moreover,  for all \te\ spaces \tes\ and Kobayashi's metrics coincide 
  (see \cite{Royden} for finite dimensional cases and  \cite{Gardiner3, Gardinerbook} for  infinite dimensional cases).

The main result  of this paper is the following theorem.
\begin{theorem}\label{maintheorem} {\it Assume  $Teich(R)$ has dimension more than $1$ and  $V$ is a tangent vector corresponding to a separating cylindrical differential.  Then there is strict inequality;}
\begin{equation}\label{inequality}
C_{Teich(R)}(\tau,(V) < Kob_{Teich(R)}(\tau,V).
\end{equation}
\end{theorem}

\vspace{.1in} 

In some ways this result is surprising partly because of  the metric equivalence implied by (\ref{equivCK}) explained in Appendix B  and partly because of  Kra's equality theorem \cite{Kra1} explained in Appendix C. That theorem shows there are many directions
in which the two infinitesimal forms for $C$ and $K$ are equal.  

To explain Theorem \ref{maintheorem} we must define a  {\it separating cylindrical differential.}   Since any simple closed curve $\gamma$ embedded in a Riemann surface divides it into one or two components, we call $\gamma$ separating if this number is two.  No matter whether $\gamma$ is separating or not we can try to maximize the modulus of a cylinder in $R$  with core curve  homotopic to 
$\gamma.$  If this modulus is bounded  it turns out that there is a unique embedded cylinder of maximal modulus. By a 
theorem of Jenkins  (\cite{Jenkins1},\cite{Jenkins2}) and Strebel  \cite{Strebelbook} it  corresponds to a unique  quadratic differential $q_{\gamma}$ which is holomorphic on $R$ with the following properties.  Its noncritical horizontal trajectories  all have length $2\pi$ in the metric $|q_{\gamma}|^{1/2},$  these trajectores fill the interior of the cylinder and no other embedded cylinder in the same homotopy class has larger modulus.   If for a separating simple closed curve such a cylinder exists and if $Teich(R)$ has dimension more than 1 then we call $q_{\gamma}$ a separating cylindrical differential.

We use the notation ${\mathcal A}_r=\{z: 1/r < |z| <r\}$ for the standard annulus which has modulus
$(1/\pi) \log r$ and ${\mathcal C}_r$ for the standard cylinder which is conformal to ${\mathcal A}_r$ and which has 
circumference $2\pi$ and height $2\log r.$  The conformal map $z \mapsto \zeta$ is realized by 
putting $$\zeta = - i\log z= - i(\log |z| + i \arg z)$$ to get a point in the infinite strip 
$$\{\zeta=\xi + i \eta: -\log r < \eta <\log r\}$$  and
then factoring the strip by the translation $\zeta \mapsto \zeta + 2 \pi.$

The rough idea for the proof of Theorem \ref{maintheorem} is to move in the direction of  a separating cylindrical differential towards points of $Teich(R)$ that represent surfaces containing tall cylinders with constant circumference.  These directions point to where anomalous aspects of the extremal length boundary of \te\ space become interesting (see for example \cite{LiuandSu}, \cite{Walsh1}) and \cite{Minsky4}.    
  By deforming in the direction of a maximal annulus ${\mathcal A}_r$ in the given homotopy class of a separating closed curve $\alpha$ in $R$  we are able to 
construct a  sequence of waist curves $\beta_n(t)$ for cylinders ${\mathcal C}_{r^n}$ lying in 
Riemann surfaces $R_n$  constructed from $R$  that have complexity roughly equal  to $n$ times the complexity of $R.$

By estimating the length of each $\beta_n$ with respect to the \te\ density 
$\lambda_{{\mathcal A}_{r^n}},$ where 
 \begin{equation}\label{basic} (1/2) \rho_{{\mathcal A}_{r^n}} \leq \lambda_{{\mathcal A}_{r^n}} \leq \rho_{{\mathcal A}_{r^n}}
 \end{equation} and where in general $\rho_R$ is Poincar\'e's density on any Riemann surface $R$ we arrive at a contradiction.  
 In \cite{GardinerLakic3} the \te\ density $\lambda_R$ is defined for any Riemann surface $R$ although we apply it only when  $R$ is conformal to one of the annuli ${\mathcal A}_{r^n}.$  Since constant (1/2) in  (\ref{basic})  is independent of $n,$  it  provides a bridge enabling  comparison of  lengths of  curves in maximal annuli ${\mathcal A}_{r^n}$ to lengths of curves in \te\ spaces $Teich(R_n).$
 
 \bigskip
 
 The \te\ density  $\lambda(p)dp$ at a point $p \in R$ is the restriction of the \te\ infinitesimal form of  \tes\ metric on $Teich(R-p)$ to the fiber of the forgetful map
 $$\Psi : Teich(R-p) \rightarrow Teich(R).$$  More detail is given in section 2. 
 
  On the one hand the right hand side of  inequality (\ref{basic}) and a computation shows that this length increases  proportionately to
 $n.$    On the other hand, the left hand inequality of (\ref{basic})  and the assumption that
           \begin{equation}\label{falseassumption} C_{Teich(R)}(\tau, V) = K_{Teich(R)}(\tau,V), \end{equation} shows that this number is bounded independently of $n.$

   This paper is divided into ten sections.  Section 1 explains a consequence 
   of the equality $$C_M(\tau,V)=K_M(\tau,V)$$ when $M=Teich(R),$ the \te\ space of a 
   Riemann surface $R$ and where  $V$ is a tangent vector to $M$ at $\tau.$  Assume $f(0) = [0] {\rm \ and \ } f(t)=[t|q|/q],$ where the square bracket 
   denotes \te\ equivalence class and $q$ is an integrable holomorphic quadratic differential on  $R.$  Also assume $df_0(1)=V$ where 
   $$\frac{\partial V}{\partial \overline{z}}= |q|/q.$$
    Equality in (\ref{falseassumption}) implies the existence of a holomorphic function $\tilde{g}$ 
   defined on a holomorphic covering of $Teich(R)$ mapping onto the unit disc $\Delta$  such that $\tilde{g} \circ f (t) =t,$ (see inequality (\ref{400})). Thus to prove Theorem \ref{maintheorem} it suffices  to show that under the given hypothesis such a function $\tilde{g}$ cannot exist.
   
Section 2 introduces Bers' fiber space $F(Teich(R))$ of $Teich(R)$ which is bianalytically 
equivalent to $Teich(R-p)$ for any point $ p \in R.$ The maps $f$ and $g$ lift to this fiber space giving 
maps $\hat{f}$ and $\hat{g}$ for which  $\hat{g} \circ \hat{f}(t) = t.$ 
See Figure 3 in section 6.

Section 3 briefly describes a natural conformal metric $\lambda_R$ on any Riemann surface  which is induced by Bers' fiber space  over $Teich(R).$ If $\rho$ is the Poincar\'e metric on $R$ then  one has the inequality (\ref{basic}).  We call $\lambda_R(p)|dp|$ the \te\ density because it is induced by the infinitesimal form of \tes\ metric restricted to the fiber over the identity in $Teich(R)$ for the forgetful map $\Psi: Teich(R-p) \rightarrow Teich(R).$ 
We only need  inequality (\ref{basic}) in the case that $R$ is an annulus but  
understanding how it is constructed is key to understanding the proof of Theorem 1. 

Section 4 estimates the decay of  \Car's infinitesimal metric $C_{{\mathcal A}_r}(1)$ where 
${\mathcal A}_r$ is the annulus $1/r<|z|<r$ as $r$ approaches infinity. 

Section 5  focuses on a quasiconformal selfmapping of ${\mathcal A}_r$ which we call $Spin_t.$ 
It spins the point $1$ on the core curve $|z|=1$ of ${\mathcal A}_r$ through a counterclockwise angle of length $t$
while holding fixed all of the boundary points of ${\mathcal A}_r$ and while shearing points along concentric circles. This quasiconformal map induces a new conformal structure on ${\mathcal A}_r$ and in this new conformal structure the annulus is conformal to an annulus of the form $A_{r(t)}$ for some annulus with outer radius $r(t)$ 
and inner radius $1/r(t).$  We give estimates for $r(t)$ from above and below in terms of $r$ and $t.$

Section 6 shows how to prove Theorem 1 in the very special case that the Riemann surface R is just an annulus and the separating simple closed curve is the core curve which separates its two boundary components.  It shows that if $V$ is the tangent vector pointing in the direction of enlarging the modulus of the annulus then
$Car_{Teich(R)}(\tau,V)$ is strictly less than $Kob_{Teich(R)}(\tau,V).$   This result is the model for the more general result in Theorem 1.  In Theorem 1 we apply the same idea to an embedded annulus ${\mathcal A}_r$ which is maximal annulus in its homotopy class.

Section 7 develops the Jenkins-Strebel theory of a maximal separating cylinder on an arbitrary Riemann surface and the associated cylindrical quadratic differential $q.$

Section 8 explains the spinning map as an element in the mapping class group of $R-p.$

Section 9 explains how to create a countable a family $R_n$ of Riemann surfaces induced by $q$ and by unrolling the cylinder along an infinite strip.

Secton 10 explains contradictory estimates of the length of waist curves that are core curves of annuli ${\mathcal A}_{r^n}$ embedded in $R_n.$  The estimates depend on inequality (\ref{basic}).

   Since writing the first version of this paper in 2015 I have learned of a paper \cite{Markovic3} by V. Markovic that proves a related result concerning the ratio of the global Carath\'eodory's and Kobayashi's metrics.

   \end{section}
    
\begin{section}{Teichm\"uller discs}

   \begin{definition}\label{tdisc}
   A \te\ disc is the image ${\mathbb D}(q)$ of any map $$f: \Delta \rightarrow Teich(R)$$ of the form \begin{equation}\label{T300} 
   f(t)=[t|q|/q]
   \end{equation}
    for $|t|<1$ where $q$ is an integrable, holomorphic, non-zero quadratic differential on 
   $R.$   
   \end{definition}
   
   By \tes\ theorem (\ref{T300}) is injective and isometric with respect to Poincar\'e's metric on $\Delta$ 
   and \tes\ metric on $Teich(R).$
    It is interesting to note that any disc isometrically embedded in a finite dimensional \te\ space necessarily has the form $f(t)$ or $f(\overline{t})$ where $f$ has \te\ form, \cite{Antonakoudis}.

    \begin{lemma}\label{ONE}Assume $V$ is a vector  tangent to the embedding of a \te\ disc given by $$f(t) = [t|q|/q]$$ where $q$ is a holomorphic quadratic differential of finite norm and suppose
    also that $C(\tau,V)=K(\tau,V).$ Then
    there exists a unique function $g \in {\mathcal G}({\tau},0)$   for which $(dg)_0(\tau,V)$ is equal to $1$ and  
 for which $(g \circ f)(t)=t$ for all $|t|<1.$ 
 
 Conversely, if there exists such a function $g$ then $C(\tau,V)=K(\tau,V).$
  \end{lemma}
 \begin{proof} Let $M(R)$ be the open unit ball of bounded Beltrami coefficients $\nu$ on $R$  with norm $$|\nu||_{\infty} = \sup_{z \in R} |\nu(z)|$$
 and let \begin{equation}\label{400}\tilde{g}(\nu)= \int \! \int_R \frac{|q|}{q}\nu dx dy .\end{equation}.
 
 Obviously $\tilde{g}$ is complex linear and therefore holomorphic on $M.$  Because the complex structure
 $Teich(R)$  is induced by local Ahlfors-Weill cross sections, $\tilde{g}$ descends to a global holomorphic function $g$ that maps the Bers' embedding of $Teich(R)$ into the unit disc.  
  By the chain rule and Schwarz's lemma
 $$(dg)_{\tau} \circ (df)_{0}(1)\leq 1.$$
 Also by Schwarz's lemma, if this is an equality then, after post composing $g$ be a rotation, we obtain  $g \circ f (t)=t$ for $|t|<1.$  Since $q$ has finite norm, it is uniquely infinitesimally extremal.  From the hypothesis that $C(\tau,V)=K(\tau,V)$ there exists a family of holomorphic functions 
 $g_n \in {\mathcal G}(\tau)$ such that $C(\tau,V)=K(\tau,V) - 1/n.$  Then the functions  $g_n \circ f$ defined on $|t|<1$ are a normal family and have a limit $g \circ f$  defined on the image of $f$ for which  
 $g \circ f (t)=t.$  Since the function $g$ determined by  $\tilde{g}$ given in (\ref{400}) is one such function it must be equal to that limit. 
 
 The reverse implication follows from the definitions (\ref{K}) and (\ref{C111}) of the infinitesimal forms $C(\tau,V)$ and $K(\tau,V).$ 
 \end{proof}

      \hspace{.1in}\end{section}
           
\begin{section}{Bers fiber space}\label{Bers}

Over every Teichm\"uller space $Teich(S)$ of a quasiconformal surface $S$ there is a canonical fiber space 
 called the Bers fiber space \cite{Bers5}.   
$$ \Psi:F(Teich(S)) \rightarrow Teich(S),$$ 
which (except in the case when $S$ is a torus) has dimension one more than the dimension of $Teich(S).$

To describe $F(Teich(S))$ we review the definition of $Teich(S).$ Assume $S$ is a fixed quasiconformal surface in the sense that $S$ is a topological Hausdorff space equipped with a system of quasiconformal charts  and assume also that this quasiconformal structure has at least one conformal substructure.  This means that among the quasiconformal charts that define $S,$ there is a subset of charts $\{z_j\}$ that satisfy  $z_j \circ (z_k)^{-1}$ is conformal for every $z_j$ and $z_k$ in the subset.
   
   The \te\ equivalence relation makes quasiconformal maps$f_0$ and $f_1(S)$ to Riemann surfaces $f_0(S)$ and $f_1(S) $ equivalent if there is a conformal map $c:f_0(S) \rightarrow f_1(S)$ and an isotopy through quasiconformal maps $h_t$ such that $h_0=f_0,$ $h_1 = c \circ h_1$ and $h_t(p)=f_0(p)=c \circ f_1(p)$ for all points in the ideal boundary of $S$ and all real numbers $t$ with $0 \leq t \leq 1.$

\begin{figure}[htp]~\label{figure0}
\hspace{.2in}\xymatrix{
& f_1(S) \ar[d]^{c}    
\\
S \ar[r]_{f_0} \ar[ru]^{f_1} & f_0(S)  }
\caption{Equivalence of $f_0$ and $f_1$}
\end{figure}

To define the Bers fibration $\Psi: F(Teich(S)) \rightarrow Teich(S)$ one picks an arbitrary point $p$ in $S$ and considers quasiconformal maps $f$ to variable Riemann surfaces $f(S).$ Since $f$ is a quasiconformal homeomorphism that maps to a Riemann surface $f(S),$ it restricts to a map 
$$f: (S-p) \rightarrow (f(S)-f(p)).$$  The equivalence relation for $F(Teich(S))$  has a similar description: two maps $f_0$ and $f_1$ from $S-p$ to $f_0(S-p)$ and  to $f_1(S-p)$ are equivalent if
there  is an isotopy $h_t$ connecting $f_0$ to $c \circ f_1$ as described above.  But now since $f(p)$ is a boundary point of  $f(S)-f(p)$ the isotopy $h_t$ must also pin down  the point $h_t(p).$  Thus the isotopy $h_t$ connecting 
$f_0$ to $c \circ f_1$ keeps track of the variable conformal structure of the marked Riemann surface $f_t(S)$  as well as the movement of the point $f_t(p)$ within that surface. 

From the above discussion we see that for an equivalence class $[[f]]$ in $F(Teich(S))$  the map $$F(Teich(S)) \ni [[f]] \mapsto f(p) \in f(S)$$ is well defined. We denote this map by $Ev$ and call it the evaluation map.  
Since $F(Teich(S))$ is simply connected $Ev$
   lifts to a map $\widetilde{Ev}$ that maps $F(Teich(S))$ into the universal cover of the Riemann surface $f(S).$

Since by definition a representative $f$ of a class $[[f]] \in Teich(R-p)$ is quasiconformal on $S-p,$ it extends uniquely to a quasiconformal map defined on $S.$  Moreover two such representatives that are isotopic through quasiconformal maps on $S-p$ are also isotopic on $S.$ Therefore the covering
 $$\Psi:F(Teich(S)) \rightarrow Teich(S),$$
   which forgets that the isotopy must fix the point at $f_t(p)$  is well-defined.  In other words, the map $$\Psi: [[f]] \mapsto [f]$$ 
   is well defined.

In general $F(Teich(S))$ is a complex manifold and for each point $\tau \in Teich(S)$ we let 
\begin{equation}\label{Ktau}
{\mathbb K}_{\tau}=\Psi^{-1}(\tau).
\end{equation}
${\mathbb K}_{\tau}$ is a one-dimensional properly embedded submanifold of $F(Teich(S))$ which is conformal to a disc.  Bers explicitly describes $F(Teich(S))$ as a moving family of normalized quasicircles.  One side of each of these quasicircles determines a point in $Teich(S)$ while the other side realizes the fiber ${\mathbb K}_{\tau}$ as a domain conformal to a disc.  The disc on one side is properly and holomorphically embedded in $F(Teich(S)).$  

Now let $R$ be the manifold $S$ for which the $\{z_j\}$ are its charts. Then denote by $\dot{R}$ the Riemann surface $R-p$ from which the point $p$ has been excised. Bers shows that the fiber space $F(Teich(S))$ with fibers equal to the fibers of  $\Psi$ is isomorphic to the fibration 
$$\Psi_p: Teich(\dot{R}) \rightarrow Teich(R),$$
where $\dot{R}=R-p.$ 
It has maximal rank at each point of its domain because
  the derivative $d\Psi_p$  of $\Psi_p$ is dual to the inclusion map of the integrable holomorphic quadratic differentials on $R$ into the integrable quadratic differentials holomorphic on $R-p.$   Except in the case when $R$ is a torus, at any point $p$ on $R$ there exists an integrable holomorphic quadratic differential $R$ with a simple pole at $p.$ This  shows that $\Psi$ has a surjective derivative at every point $\dot{\tau} \in F(Teich(R))$ except when $R$ is a torus.

        Bers gives a global chart for $Teich(\dot{R})$ by representing it as a slice in the quasi-Fuchsian 
        space of a Fuchsian covering group $G$ of $R.$       He forms the 
        quasi-Fuchsian groups $$w^{\mu} \circ G \circ (w^{\mu})^{-1}$$ 
        where $\mu$ is a Beltrami coefficient compatible $G$ in the exterior of $\Delta$ and is identically equal to zero in its interior.
        The quotient of the action of $G$ on the interior represents $R$ and the quotient by its action on the exterior represents $R_{\tau}$ where $\tau= \Psi(\dot{\tau}).$         
                       
                     In the following lemma we assume $V$ is a tangent vector to $Teich(R)$ of the form 
           $\overline{\partial} V = |q|/q,$ where $q$ is a quadratic differential form supported in the complement of $\Delta$ and where $\overline{\partial} V = 0$ in  $\Delta.$  Note that if $q$ is integrable and holomorphic on $R$ then it is also integrable and holomorphic on $\dot{R}.$  Without changing the notation we  view $V$ as  representing a vector tangent to $F(Teich(R))$ and at the same time a vector tangent to $Teich(R).$
           \begin{lemma}\label{lemma2} Suppose $\tau \in Teich(R-p)$ and 
 $\overline{\partial}V =|q|/q$ where $q$ is a holomorphc quadratic differential on $R.$ In the setting just described suppose $$C_{Teich(R)}(\Psi(\tau),V)=K_{Teich(R)}(\Psi(\tau),V).$$
           Then $C_{F(Teich(R))}(\tau,V)=K_{F(Teich(R))}(\tau,V).$
           \end{lemma}
\begin{proof} The \te\ disc $f(\Delta) =\{[t \frac{|q|}{q}]: |t|<1\} \subset Teich(R)$
lifts to a \te\ disc $\{[[t \frac{|q|}{q}]]: |t|<1\}$ in $Teich(\dot{R}). $  If we denote the corresponding lifted mapping by $\hat{f}$ and let $\hat{g}= g \circ \Psi$ 
where $\Psi$ is the projection from 
Bers fiber space $F(Teich(R))$ to  $Teich(R)$ and where $g$ is  referred to in Lemma \ref{ONE},
we obtain
$$\hat{f}: \D \rightarrow F(Teich(R)) {\rm \ and \ }  \hat{g}: F(Teich(R)) \rightarrow \D$$  where both $\hat{f}$ and $\hat{g}$ are
holomorphic  and \begin{equation}\label{hats} \hat{g} \circ \hat{f}(t) =t.\end{equation}
The conclusion follows.
\end{proof}

  \end{section}

  \begin{section}{Comparing metrics on surfaces}
  For every $\tau \in Teich(R)$ the fiber ${\mathbb K}_{\tau}=\Psi^{-1}(\tau)$ of the forgetful map,
   \begin{equation}\label{forget2} \Psi:F(Teich(R)) \rightarrow Teich(R),
   \end{equation}
   carries two metrics that are induced by  inherent geometry.   The first is the Poincar\'e metric $\rho$
   defined in (\ref{one}).  It is natural because ${\mathbb K}_{\tau}$ is conformal to a disc.   When viewed as 
   a metric on ${\mathbb K}_{\tau}$ we denote it 
   by $\rho_{\tau}$ and when viewed as a metric on $R$ we denote it by $\rho_R.$
      The second metric is the restriction of the  infinitesimal form  of \tes\ metric on $F(Teich(R))$ to ${\mathbb K}_{\tau}.$  Similarly, when viewed as a metric on ${\mathbb K}_{\tau}$ we denote it by $\lambda_{\tau},$ and when viewed as a metric on $R=R_{\tau}$ we denote it by $\lambda_{R}.$  In \cite{GardinerLakic3} we have called $\lambda_{R}$ the \te\ density.
     
   \begin{definition}\label{tdensity} $\lambda_{\tau}(z)|dz|$ is the pull-back by a conformal map $c:\Delta \rightarrow {\mathbb K}_{\tau}$  of the infinitesimal form of  \tes\ metric on $F(Teich(R))$
 restricted to    ${\mathbb K}_{\tau}.$   \end{definition}
 
   From the previous section we know that $F(Teich(R))$ has a global coordinate given by 
   $Teich(R-p),$  
   where $R$ is a marked Riemann surface and $p$ is a point on $R.$
   For any point $[[g]] \in Teich(R-p)$ the evaluation map $Ev$ is defined by $Ev([[g]])=g(p)$ is holomorphic 
   and, since the domain of $Ev$ is simply connected, $Ev$ lifts to $\Delta$ by the covering map $\Pi_p$ to $\widetilde{Ev}$ mapping $Teich(R-p)$ to 
   $\Delta.$  Note that the target of the map $Ev$ is the Riemann surface $R_{\tau}$ and the target of $\widetilde{Ev}$ is the universal covering of $R_{\tau}.$
For this reason we can interpret the inequality of the next theorem as an inequality of conformal metrics on $R=R_{\tau}.$
Also, if a quasiconformal map $f$ represents an equivalence class in $Teich(R_{\tau}-p)$  
and if $f$ represents a point in ${\mathbb K}_{tau}$ defined by (\ref{Ktau}), then $f(p)$ is a point in $R_{\tau}$ since in that case $f(R_{\tau})=R_{tau}.$

The following theorem is proved in \cite{GardinerLakic3}.
   
\begin{theorem}\label{thm2} For any surface $R$ with marked conformal structure $\tau$ whose universal covering is conformal to $\Delta,$  the conformal metrics $\rho_{\tau}$ and $\lambda_{\tau}$ satisfy
\begin{equation}\label{equthm2}(1/2) \rho_{\tau}(p)|dp| \leq \lambda_{\tau}(p)|dp| \leq \rho_{\tau}(p)|dp|.\end{equation}
\end{theorem}

This theorem has a companion theorem that expresses $\lambda_{\tau}$ in terms residues of quadratic quadratic differentials holomorphic on $R-p$  that have integral norm less than or equal to $1.$  For the statement see  Appendix D. 
In the proof of the main theorem of this paper, that is, Theorem \ref{maintheorem}, we use Theorem  \ref{thm2} only for the case  $R$ is an annular domain conformal to  ${\mathcal A}_r = \{z : 1/r < |z| <r\}.$  
\end{section}
        
\begin{section}{The $C$ and $K$ metrics on an annulus}

Let $C_r$ and $K_r$  be  the infinitesimal forms of  \Car's and Kobayashi's metrics on  the annulus $${\mathcal A}={\mathcal A}_r=\{w: (1/r) < |w| <r\}.$$ We identify the tangent space of the plane domain ${\mathcal A}_r$ with direct product ${\mathcal A}_r \times {\mathbb C}. $ Since ${\mathcal A}_r$ is acted on by the continuous group of rotations the value of either of these forms is invariant along the core curve  $\{w: |w|=1\}$ and we focus only on the values of  $K_{{\mathcal A}_r}(1,1)$ and $C_{{\mathcal A}_r}(1,1),$
which are defined by the general formulas (\ref{K}) and (\ref{C111}) in the introduction.
Also, we shorten these notations to $C_{r}(1)$ and $K_{r}(1).$
 
     The logarithm provides a conformal map from ${\mathcal A}_r$ to a cylinder which we can think of as being vertical with the circumference equal to $2 \pi$ and height equal to $2\log r.$ With the same circumference it becomes taller as $r$ becomes larger. 
      In the sequel the only results we need from this section are inequality (\ref{FOUR})  in the following lemma  and formula (\ref{K11}).
      
\begin{lemma}\label{Car'smetric} Assume $r>1$ and $n$ is a positive integer.  Then 
\begin{equation}\label{FOUR}C_{r^n}(1) \leq \frac{1}{n}  \cdot \frac{r}{r-1}.
\end{equation}
\end{lemma}
\begin{proof} 

We assume the supremum in (\ref{C111}) for the tangent vector $1$ at the point $1$ in ${\mathcal A}_{r^n}$ is realized by $|\tilde{g}'(1)|$ where $\tilde{g}(1)=0$ and $\tilde{g}$ is a holomorphic function with domain ${\mathcal A}_{r^n}$ and range contained in the unit disc $\Delta.$
If we put $g(z) = \tilde{g}(z^n),$ then $g$ is  holomorphic and defined on ${\mathcal A}_r$ with $g(1)=0$ and has the same range as $\tilde{g}.$  Since $C_{{\mathcal A}_{r}}(1)$ is the supremum of the values of $|g'(1)|$ for all such functions $g$, we have
\begin{equation}\label{180}C_{{\mathcal A}_{r^n}}(1) = |\tilde{g}'(1)| = (1/n) \cdot |g'(1)|  \leq (1/n) \cdot C_r(1).
\end{equation}

Any holomorphic function $g$ mapping ${\mathcal A}_r$ into the unit disc with $g(1)=0$ restricts to a map of a disc of radius $1-1/r$ centered at the point $1.$  By Schwarz's lemma $|g'(1)|$ satisfies 
$$ |g'(1)| < \frac{1}{1-1/r} = r/(r-1).$$
This proves the lemma.
\end{proof}

It is important to compare the rates of decay of $C_r$ and $K_r.$ 
Because the formula for $K_r(1)$ is much more simple, this comparison is easy to see with the benefit of the previous lemma. Using the logarithmic coordinate $\zeta = - i \log w$ for ${\mathcal A}_r$, the core curve $z=e^{i\theta}$ is covered by the real axis $\eta=0$ and  ${\mathcal A}_r$ is conformal to the rectangle 
$$ \{ \zeta = \xi + i \eta: \log(1/r) < \eta < \log r, \ \ -\pi \leq \xi < \pi \},$$
with the two sides at $\xi = -\pi$ and $\xi = \pi$ identified by the translation 
$\zeta \mapsto \zeta +2 \pi.$ Then the Kobayashi (Poincar\'e) metric is 
\begin{equation}\label{K1}K_r(\zeta,V) =\frac{\pi|d\eta(V)|}{(2\log r)\cos (\frac{\eta \pi}{2\log r}) }.
\end{equation} 
In the $\zeta$-coordinate along the core curve $\eta=0$ where the cosine is equal to $1$ we have
\begin{equation}\label{K11}
K_r(1)=\frac{\pi}{2 \log r}.
\end{equation}

Although it is unnecessary for the proof of  Theorem \ref{maintheorem}, the following explicit formula is given by Simha \cite{Simha} 
\begin{equation}\label{C1}
C_r(w)(V)|dw(V)| = \frac{2}{r} \cdot
\frac{\Pi_1^{\infty}(1+r^{-4n})^2(1-r^{-4n})^2}{\Pi_1^{\infty}(1+r^{-4n+2})^2(1-r^{-4n+2})^2}|dw(V)|.
\end{equation}

Of course from inequality (\ref{C<K}) 
 for any tangent vector $V$ and any point $P$ in any complex manifold $M,$
 the ratio $C_M(P,V)/K_M(P,V) \leq 1.$
From  formulas (\ref{K11}) and (\ref{C1}) when $M$ is the annulus ${\mathcal A}_r$ we have the more precise formula valid for any point $P$ along the core curve, namely, 
the ratio \begin{equation}\label{pureratio}
C_r(|w|=1,V)/K_r(|w|=1,V) = \frac{4}{\pi} \cdot \frac{\log r}{ r} \cdot 
\frac{\Pi_1^{\infty}(1+r^{-4n})^2(1-r^{-4n})^2}
{\Pi_1^{\infty}(1+r^{-4n+2})^2(1-r^{-4n+2})^2}.
\end{equation}
    In particular, the graph of this ratio is shown in Figure 2 and was provided to me by Patrick Hooper and Sean Cleary.  It shows that the ratio is always less than $1$ and  monotonically decreases to $0$ as  $r$ increases to  $\infty.$   
\begin{figure}[htp]~\label{figure3}    
\centering{
\includegraphics[scale=0.6]{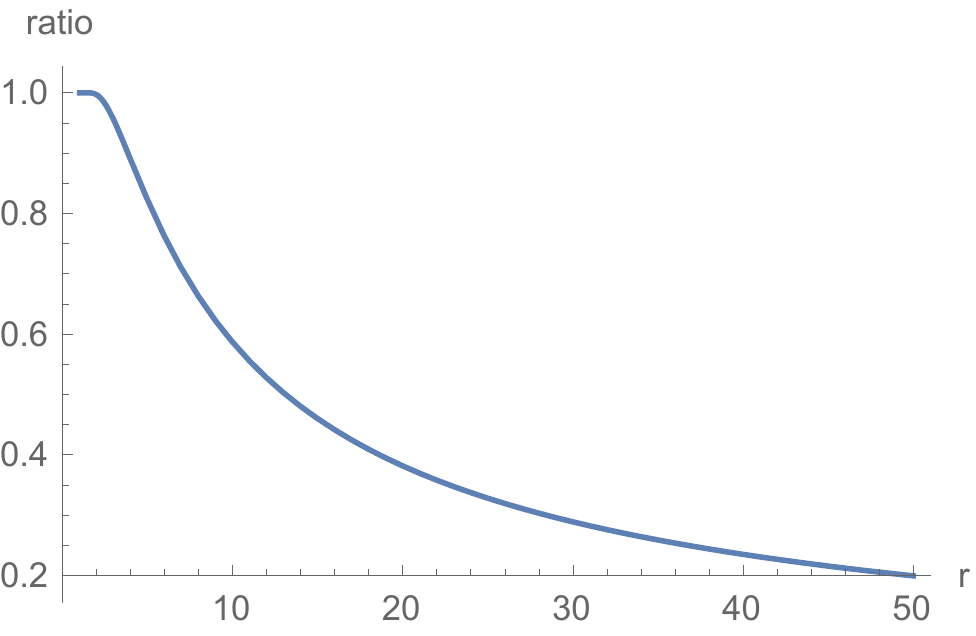}}
\caption{The metric ratio}\end{figure}

 \end{section}

 %\begin{figure}[htp]~\label{figure3}
%\centering{
%\includegraphics[scale=0.6]{ratio2.pdf}}
%\caption{metric ratio}\end{figure} 

%\end{figure}

%\begin{figure}[htp]~\label{figure3}
%\begin{centering}
%\includegraphics[scale=0.6]{ratio2.pdf}
%\caption{ metric ratio}
%\end{centering}
%\end{figure}

%\begin{figure}[htp]\label{figure4}
%\hspace{.2in}\xymatrix{
%& Teich({\mathcal A}_r-\{1\}) \ar[d]^{\Psi}  \ar[dr]^{\hat{g}}    
%\\
%\{s: |s|<1\} \ar[r]_f \ar[ru]^{\hat{f}} &
%Teich({\mathcal A}_r)  \ar[r]_g& 
 %\{s : |s|<1\} }
%\caption{$\hat{f}, \hat{g}, f $ and $g$}
%\end{figure}

\begin{section}{Spinning an annulus}\label{s5}

We  are interested in a particular selfmap $Spin_t$ of the annulus 
$${\mathcal A}_r=\{z: 1/r<|z|<r\}.$$ $Spin_t$ spins counterclockwise the point  $p=1$ on the unit circle
$$\{z: |z|=1\} \subset {\mathcal A}_r$$ to the point $p=e^{i t}$ for real numbers $t$ while keeping points on the boundary of ${\mathcal A}_r$ fixed.   The formula for $Spin_t$ is easier to write 
in the rectangular coordinate $\zeta$ where we put 
$$\zeta = \xi + i \eta  = -i \log z.$$
We call $\zeta$ rectangular because the strip 
\begin{equation}\label{strip1}
\{\zeta=\xi+i\eta:-\log r \leq \eta \leq +\log r\}
\end{equation}
 factored by the translation $\zeta \mapsto \zeta + 2 \pi$ is a conformal realization ${\mathcal A}_r$ with a fundamental domain that is the rectangle 
$$\{\zeta:0 \leq \xi \leq 2 \pi {\rm \ and  } -\log r \leq \eta \leq +\log r  \}.$$
The covering map is $\Pi(\zeta)=z=e^{-i\zeta}$  and the points $\zeta=2 \pi n$ for integers $n$ cover the point $z=1.$
In the $\zeta$ coordinate we make the following definition.

  \begin{definition}\label{spin}    For $\zeta = \xi +i \eta$ in the strip $-\log r \leq \eta \leq \log r,$ let 
  \begin{equation}\label{inv}Spin_t(\zeta) = \xi + t \left(1-\frac{|\eta|}{\log r}\right)+ i \eta.
  \end{equation} 
  \end{definition}

Note that  $Spin_t$ shears points with real coordinate $\xi$ and with imaginary coordinate $\pm \eta$ along horizontal lines to points with real coordinate $\xi + t(1-|\eta|/\log r)$ while fixing points on the boundary of the strip where $\eta=\pm \log r.$

%In the $\zeta$-coordinate, where $\zeta = \xi + i \eta,$
%$$\mu_t = \left\{\begin{array}{lll}
%z+\epsilon h(z) & {\rm \ in  \ } N\\
%z & {\rm \ outside \ of \ } N,
%\end{array}\right. $$

We let $\mu_t$ denote 
the Beltrami coefficient of $Spin_t$
and let ${\mathcal A}_r(\mu_t)$ be the annulus ${\mathcal A}_r$ with conformal structure determined by 
 $\mu_t.$ 
By the uniformization theorem applied to the annulus ${\mathcal A}_r(\mu_t)$ we know that there is a positive number $r(t)$ such that ${\mathcal A}_r(\mu_t)$ is conformal to $A_{r(t)}.$

\begin{theorem}\label{thm4} Assume $t>2 \pi.$ Then the unique real number $r(t)$ with the property that ${\mathcal A}_r(\mu_t)$ is conformal to $A_{r(t)}$ satisfies 
    \begin{equation}\label{250}
    t \cdot \frac{(1-2\pi/t)^2+ ((\log r)/t)^2}{(\sqrt{1+ ((\log r)/t)^2})} \leq \frac{\log r(t)}{\log r} 
   \leq 
   t \cdot \left(\frac{1}{t^2} + \frac{1}{(\log r)^2}\right)^{1/2}.
   \end{equation}

\end{theorem}
\begin{proof}
We begin the proof with a basic inequality valid for any pair of measured foliations 
$|du|$ and $|dv|$ 
on any orientable surface $S$ with a marked conformal structure $\tau.$
Actually, this inequality is true for any pair of weak measured foliations. In this paper we only need the result for measured foliations.  See \cite{Gardiner7}
for the definition of a weak measured foliation.

We denote by $S(\tau)$ the surface $S$ with the conformal structure $\tau.$  
For any smooth surface the integral $\int\!\int_S |du \wedge dv|$ is well defined for any pair of weak measured foliations $|du|$ and $|dv|$ on defined $S.$  The Dirichlet integral of a measured foliation 
$|du|$ defined by
 $$Dir_{\tau}(|du|)= \int \! \int_{S(\tau)} (u_x^2+u_y^2) dx dy$$
 is well defined only after a complex structure $\tau$ has been assigned to $S.$ $\tau$ gives invariant meaning to the form $(u_x^2 + u_y^2)dx \wedge dy$ where $z=x + i y$ is any holomorphic local coordinate.
 The following is a version of the Cauchy-Schwarz inequality.
\begin{lemma}\label{Five}  With this notation
\begin{equation}\label{thirtyone}
\left(\int \! \int_S |du \wedge dv|\right)^2  \leq \int\!\int_{S(\tau)}(u_x^2+u_y^2) dxdy  \int\!\int_{S(\tau)}(v_x^2+v_y^2)dxdy.
\end{equation}
\end{lemma}
\begin{proof}See \cite{Gardiner7}.
      \end{proof}

The measured foliations $|d\eta|$ and $|d\xi|$ on the strip (\ref{strip1}) induce horizontal and vertical foliations whose corresponding trajectories are concentric circles and radial lines on ${\mathcal A}_r.$  Correspondingly,
$|d\eta\circ (Spin_t)^{-1}|$ and $|d\xi \circ (Spin_t)^{-1}|$ are horizontal and vertical foliations on ${\mathcal A}_r(\mu_t).$  We apply Lemma \ref{Five} to these foliations 
with $S={\mathcal A}_r(\mu_t)$ and $S(\tau)={\mathcal A}_r$ and we get
$$(4 \pi \log r(t))^2 \leq 
\left[4 \pi \log r\right]\left[(4 \pi \log r)(1 + \frac{t^2}{(\log r)^2})\right],$$
which leads to the right hand side of (\ref{250}).

To obtain the left hand side we work with the conformal structure on ${\mathcal A}_r(\mu_t)$ where 
the slanting level lines of $\xi \circ (Spin_t)^{-1}$ are viewed as orthogonal to the vertical 
lines of $\eta.$ By definition the number $r(t)$ is chosen so that $A_{r(\mu_t)}$ is conformal  
$A_{r(t)}$ and this modulus is equal to 
$$\frac{2\log r(t)}{2 \pi}.$$ 
On the other hand this modulus is bounded below by any of the fractions

$$\frac{(\inf_{\beta} \int_{\beta} \sigma)^2}{area(\sigma)},$$
where $\beta$ is any arc that joins the two boundary contours of the annulus
and $\sigma$ is any metric on $A_{r(t)}.$

From the definition of extremal length and using the metric $\sigma = |d\zeta|$ where $\zeta$ is the strip domain parameter one obtains
$$\frac{t}{2} \leq \frac{\log r(t)}{\log r},$$
The lemma follows since the left hand side of (\ref{250}) is smaller than $\frac{t}{2}$ for $t>2\pi.$
\end{proof} 

\end{section}

\begin{section}{Comparing $C$ and $K$ on $Teich({\mathcal A}_r)$}\label{annuli}

\begin{theorem}\label{thm5}
For every point $\tau \in Teich({\mathcal A}_r)$ and for the tangent vector $V$ with 
$\overline{\partial} V=|q|/q$ where $q=\left(\frac{dz}{z}\right)^2,$
\begin{equation}\label{e6} C_{Teich({\mathcal A}_r)}(\tau,V) < K_{Teich({\mathcal A}_r)}(\tau,V).\end{equation}
\end{theorem}
\begin{proof}
 By \tes\ theorem the embedding  
 $$\{s:|s|<1\} \ni s \mapsto  f(s)= [s|q|/q] \in Teich({\mathcal A}_r)$$
 is isometric in \tes\ metric and it lifts to an isometric embedding
  $$\{s:|s|<1\} \ni s \mapsto \hat{f}(s)= [[s|q|/q]] \in Teich({\mathcal A}_r-\{1\}).$$
 \begin{figure}[htp]\label{figure4}
\hspace{.2in}\xymatrix{
& Teich({\mathcal A}_r-\{1\}) \ar[d]^{\Psi}  \ar[dr]^{\hat{g}}    
\\
\{s: |s|<1\} \ar[r]_f \ar[ru]^{\hat{f}} &
Teich({\mathcal A}_r)  \ar[r]_g& 
 \{s : |s|<1\} }
\caption{$\hat{f}, \hat{g}, f $ and $g$}
\end{figure}

   The map $f$  in Figure 3 is defined by $f(s)=[s|q|/q]$ and $\hat{f}(s)=[[s|q|/q]]$ 
   where the single and double brackets denote equivalence classes of Beltrami coefficients representing elements of  
   $Teich({\mathcal A}_r)$ and $Teich({\mathcal A}_r-\{1\}),$ respectively.  The map $g$ is any holomorphic map from  $Teich({\mathcal A}_r)$ into $\Delta$ and $\hat{g} = g \circ \Psi$  where $\Psi$ is the forgetful map.  That is, $\Psi$ applied to an equivalence class $[[\mu]]$ 
   forgets the requirement that the isotopy in the equivalence must pin down the point $1$ in ${\mathcal A}_r$  while keeping the requirement that it must pin down all of the points on both the innner and outer boundaries of 
   $ {\mathcal A}_r.$

   Just as in (\ref{strip1}) we use the  coordinate $\zeta = -i \log z$  that realizes the annulus ${\mathcal A}_r$  as the quotient space of 
   $$ Strip = \{\zeta=\xi + i\eta: |\eta| < \log r\} $$
   factored by the translation $\zeta \mapsto \zeta+2 \pi.$
In the coordinate $\zeta$ the quadratic differential $q$ and the tangent vector $V$ have  simple expressions:
 $$q = (d\zeta)^2 {\rm \ and \ } \overline{\partial} V =  \frac{d \overline{\zeta}}{d \zeta}.$$
    
      For any positive integer $n$ we let ${\mathcal A}_{r^n}$ be the same strip factored  by the translation 
      $\zeta \mapsto \zeta + 2 \pi n$ and consider the core curves $\beta_n$ that are covered 
      by the interval $\eta=0, \ 0 <\xi \leq 2 \pi n$ in the strip. Denote by $\lambda_n,c_n {\rm \ and \ } 
      \rho_n$ the \te, Carath\'eodory and Kobayashi densities on the surface $R={\mathcal A}_{r^n}.$
      Now use Theorem 2 and the assumption that there is equality in (\ref{e6}) to  derive 
      contradictory estimates for the $\lambda_n(\beta_n),$ that is, the $\lambda_n$-length of $\beta_n$ in 
      the annulus ${\mathcal A}_{r^n}.$
            Since $\lambda_n > (1/2)\rho_n$  from (\ref{K1}) we have 
      \begin{equation}\label{201}\lambda_n(\beta_n) > (1/2)\rho_n(\beta_n)= n \cdot \frac{\pi}{4 \log r}.
      \end{equation}
      
      On the other hand since we assume inequality (\ref{e6}) is actually an equality
      there exists a holomorphic function $g_n:Teich({\mathcal A}_{r^n}) \rightarrow \Delta$ such that
      $g_n \circ f_n(s)=s$ where $f_n(s) = [s|q_n|/q_n]$ and $q_n = (d \zeta)^2$ in the same strip.
      Moreover $\hat{g}_n \circ \hat{f}_n(s)=s$ where $\hat{f}_n(s)=[[s|q_n|/q_n]]$ and
      where $\hat{g}_n = g_n \circ \Psi.$
      $\hat{g}_n$ is holomorphic and automorphic for the translation $\zeta \mapsto \zeta + 2 \pi n$ so its restriction the \te\ disc $[[s|q_n|/q_n]]$ determines a function on the annulus for which $\hat{g}_n \circ \hat{f}_n (s)=s.$ This restriction must realize the extremal value for the extremal problem in (\ref{C111}) that defines $c_n.$  From Theorem 2 and Lemma \ref{Car'smetric}  the $\lambda_n$-lengths of the curves $\beta_n$ are bounded independently of $n$ because 
      $$\lambda_n(\beta_n) \leq \rho_n(\beta_n)=c_n(\beta_n) \leq n \cdot (1/n) \cdot \frac{r}{r-1}=\frac{r}{1-r}, 
      $$
      which contradicts (\ref{201}).
      
\end{proof}

  \end{section}

\begin{section}{Maximal separating cyclinders}\label{separating}

Let $R$ be a Riemann surface for which $Teich(R)$ has dimension at least $2$ and assume  $R$ is of finite analytic type, by which we mean that the fundamental group of $R$ is finitely generated.
Consider a simple closed curve $\gamma$  that divides $R$ into two connected components $R_1$ and $R_2$ and assume that both $R_1$ and $R_2$ have non-trivial topology in the sense that each of them has a fundamental group generated by two or more elements.
From a theorem of Jenkins \cite{Jenkins1} and also of Strebel \cite{Strebelbook}, there exists on $R$ an integrable, holomorphic, quadratic differential $q_{\gamma}(z)(dz)^2$ that realizes $R$ in cylindrical form, in the following sense.
A cylinder realized as the factor space of a horizontal strip in the $\zeta$-plane
 \begin{equation}\label{cylinder1} {\mathcal C} = \{\zeta=\xi + i\eta:  - \log r \leq \eta \leq \log r \}/ (\zeta \mapsto \zeta + 2 \pi)
 \end{equation}
  with $r>1$ is embedded in $R$ with the following properties:
\begin{enumerate}[i)]
\item{$(d \zeta)^2 = - (d \log z)^2$ is the restriction of a global holomorphic quadratic differential  $q_{\gamma}(z)(dz)^2$ on $R$ with $\int \int_R |q| dx dy = 2 \pi \log r,$ }
\item{the horizontal line segments in the $\zeta$-plane that fill the rectangle 
$\{\zeta=\xi + i\eta: 0 \leq \xi < 2 \pi,  \xi < 2 \pi, - \log r < \eta < \log r \}$
 comprise all of the regular horizontal closed trajectories of $q$ on $R,$ }
\item{the core curve $\gamma =\{\zeta=\xi + i\eta: \eta= 0, 0 \leq \xi \leq 2 \pi\}$ separates $R$ into two components $R_1$ and $R_2,$}
\item{the remaining horizontal trajectories of $q$ on $R,$ that is, those which run into singular points of $q,$ comprise two critical graphs, one lying in the subsurface $R_1$ and the other lying in the subsurface $R_2,$}
\item{the critical graphs $G_1 \subset R_1$ and $G_2 \subset R_2$ include as endpoints all the punctures of $R$ as simple poles of $q,$ and}
\item{each of the graphs $G_1$ and $G_2$ is connected,}
\item{for every closed curve $\tilde{\gamma}$  homotopic in $R$ to the core curve $\gamma$
$\int_{\tilde{\gamma} } |q_{\gamma}|^{1/2} \geq \int_{\gamma} |q_{\gamma}|
^{1/2}.$ }
\end{enumerate}

\begin{theorem}{[Jenkins and Strebel]}\label{thm6} The condition that the embedded cylinder described above has maximal modulus in its homotopy class implies  that $q_{\gamma}$  is a global holomorphic quadratic differential on $R$ whose restriction to the interior of the cylinder is equal to $(d\zeta)^2.$  This $q_{\gamma}$ is maximal in the sense that any  quadratic differential $q$ holomorphic on $R$ satisfying $||q||=||q_{\gamma}||$ and the inequality 
$\int_{\tilde{\gamma} } |q|^{1/2} \geq \int_{\gamma} |q_{\gamma}|^{1/2}$ 
for every $\tilde{\gamma}$ homotopic to the core curve $\gamma$
is identically equal to $q_{\gamma}.$ Moreover, the critical graphs $G_1$ and $G_2$ are connected.
\end{theorem}
\begin{proof} This theorem is proved in \cite{Strebelbook} except for the part concerning the components of the critical graph $G_1$ and $G_2.$
To see that $G_1$ is connected take any two points in $G_1.$  Since $R$ is arcwise connected, by assumption they can be joined by an arc in $R.$ If that arc crosses the core curve $\gamma,$ cut off all the parts that lie in $R_2$ and join the endpoints by arcs appropriate subarcs of  the core curve $\gamma$ to obtain a closed curve $\alpha$ in $R_1 \cup \gamma$ that joins the two points.  Then use the vertical trajectories of the cylinder to project every point of $\alpha$ continuously to a curve that lies in the critical graph $G_1$ that joins the same two points.  This shows that $G_1$ is arcwise connected. Of course, the same is true for $G_2.$

That the two conditions $||q||=||q_{\gamma}||$ and the inequality 
$\int_{\tilde{\gamma} } |q|^{1/2} \geq \int_{\gamma} |q_{\gamma}|^{1/2}$ imply that $q=q_{\gamma}$ follows from Gr\"otzsch's argument or from the minimum norm principle proved in \cite{Gardinerbook}. Finally, if 
$(d\zeta)^2$ failed to be the restriction of a global holomorphic quadratic differential on $R,$ it would be possible to deform the embedding of the cylinder inside $R$ so as to increase the modulus of the cylinder. 
\end{proof}

\begin{definition} We call a cyclindrical differential $q_{\gamma}$ with the properties of the previous theorem {\it separating} and any simple closed curve homotopic to $\gamma$ is also called a  {\it separating curve}.
\end{definition}
  
 \end{section}

   \begin{section}{Spinning on an arbitrary surface}

We extend the estimates for moduli of annuli given in Theorem \ref{thm4} to moduli of a maximal embedded separating cylinder ${\mathcal C} \subset R.$   Just as in Theorem 4 we let $\mu_t$ be the Beltrami coefficient of $Spin_t.$ 
 Our objective is to estimate the modulus of the embedded annulus with this new conformal structure.  We denote this new annulus by ${\mathcal A}(\mu_t).$  
 
The measured foliations 
$|d \xi \circ Spin_t^{-1}|$ and $|d \eta \circ Spin_t^{-1}|$ on 
$R-p(t),$  where $p(t)$ is a point on the core curve of a maximal cylinder 
${\mathcal C} \subset R,$ is represented by the point $t$ on the real axis in the rectangular strip
coordinate shown in (\ref{strip1}). In this way $\mu_t$ determines a conformal structure on $R$ and $R-p(t).$ 

 \begin{definition}\label{r(t)} The real number $r(t)$ is the number for which the standard annulus 
 ${\mathcal A}_{r(t)}$ with inner and outer radii equal to  $1/r(t)$ and $r(t)$ is conformal to ${\mathcal A}(\mu_t)$ where $\mu(t)$ is viewed as a Beltrami coefficient on $R.$
 \end{definition}

\begin{theorem}\label{thmsix}  The unique real number $r(t)$ with the property that ${\mathcal A}_r(\mu_t)$ is conformal to $A_{r(t)}$ satisfies 
    \begin{equation}\label{300}
    t \cdot \frac{(1-2\pi/t)^2+ ((\log r)/t)^2}{(\sqrt{1+ ((\log r)/t)^2})} \leq \frac{\log r(t)}{\log r} 
   \leq 
   t \cdot \left(\frac{1}{t^2} + \frac{1}{(\log r)^2}\right)^{1/2}.
   \end{equation}

\end{theorem}
\begin{proof}

Consider the  the
measured foliations 
$|d\eta\circ (Spin_t)^{-1}|$ and $|d\xi \circ (Spin_t)^{-1}|,$ which are horizontal and vertical foliations on the annulus ${\mathcal A}_r(\mu_t)$ which has conformal structure given by $\mu_t.$  Applying Lemma \ref{Five} to these foliations 
with $S={\mathcal A}_r(\mu_t)$ and $S(\tau)={\mathcal A}_r$ and we get
$$(4 \pi \log r(t))^2 \leq 
\left[4 \pi \log r\right]\left[(4 \pi \log r)(1 + \frac{t^2}{(\log r)^2}\right],$$
which leads to the right hand side of (\ref{300}).

To obtain the left hand side we work with the conformal structure on ${\mathcal A}_r(\mu_t)$ where 
the slanting level lines of $\xi \circ (Spin_t)^{-1}$ are viewed as orthogonal to the vertical 
lines of $\eta.$ By definition the number $r(t)$ is chosen so that $A_{r(\mu_t)}$ is conformal  
$A_{r(t)}$ and so the modulus of this annulus is equal to 
$$\frac{2\log r(t)}{2 \pi}.$$ 
This is bounded below by any of the fractions
$$\frac{1}{4} \cdot \frac{(\inf_{\beta} \int_{\beta} \sigma)^2}{area(\sigma)},$$
where $\beta$ is any homotopy class of closed curve on $R$  that joins the two boundary contours of the annulus and has intersection number $2$ with the core curve of $A_{r(t)}.$
We may take $\sigma$ to be the flat $\sigma(\zeta)|d\zeta|$ on $A_{r(t)}$  given by 
$\sigma({\zeta}) \equiv 1.$  Then $$\int_{\beta}\sigma(\zeta) |d\zeta| =\int_{\beta} |d\zeta| \geq 4 \log r(t).$$
From the definition of extremal length one obtains
$$\frac{\left(2\sqrt{(t-2\pi)^2+ (\log r)^2}\right)^2}{4 \left(\sqrt{t^2+ (\log r)^2}\right)
\log r} \leq \frac{\log r(t)}{\log r},$$
This lower bound leads to the left hand side of (\ref{300}).
\end{proof}

   \end{section}

\begin{section}{Unrolling}

Just as in the previous sections we let ${\mathcal C}_{\gamma}$ be the maximal closed cylinder  corresponding to a separating simple closed curve $\gamma$ contained in $R.$ Then 
$$Strip_{\gamma}/(\zeta \mapsto \zeta+2\pi) \cong {\mathcal C}_{\gamma}.$$ 
Moreover the interval $[0,2\pi]$ along the top perimeter of ${\mathcal C}_{\gamma}$ partitions into pairs of subintervals of equal length that are identified by translations or half translations of the form 
      $$ \zeta \mapsto \pm \zeta + constant$$
      that realize the conformal structure of $R$ along the top side of 
            ${\mathcal C}_{\gamma}.$
            
            Similarly the bottom perimeter of ${\mathcal C}_{\gamma}$ partitions into pairs of intervals similarly identified that realize the local consformal structure of $R$ on the bottom side.
 
 In this section there is a fixed value of $r$ which is determined by separating curve and the conformal structure of $R.$  The value of $r$ for annulus ${\mathcal A}_r$ is that value that makes that 
 ${\mathcal A}_r$ conformal to the annulus in $R$ that maximizes this modulus.  But we also consider annuli  ${\mathcal A}_{r^n}$ that are maximal annuli in $R_n$ where $R_n$ is obtained from the same strip factored by the translation $\zeta \mapsto \zeta + 2 \pi n.$ Thus $R=R_1$ and the same Beltrami coeffient $\mu$
 determines a well-defined conformal structure on each surface $R_n.$ Our task is to estimate the modulus of the embedded annulus in $R_n$ and in $R_n - p.$  There is one case when our estimate fails.  That is when $R$ is the Riemann sphere minus four points and the separating curve is one that separates two of these from the other two.  Then the \te\ space $T(R)$ has dimension 1 and the comparison of the moduli for the annuli on $R_n$ and on $R_n-p$ breaks down.
 We let ${\mathcal C}_n$ be the cylinder a conformal to ${\mathcal A}_{r^n}.$

 \begin{lemma}\label{lemmaeight} Assume $\gamma$ is a separating closed curve in $R$ and let $q_n$ be the quadratic differential $(d\zeta)^2$ on ${\mathcal C}_n$ and $\Delta= \{t:|t|<1\}.$  Then 
 $$\Delta \ni t \mapsto [f^{t|q_n|/q_n}]_{R_n} \in Teich(R_n)$$ is an injective isometry.  
 \end{lemma} 
  \begin{proof} Since the domain ${\mathcal A}_{r^n}$ has extremal modulus in $R_n$ the quadratic differential $q_n$ on ${\mathcal A}_{r^n}$ is the restriction of an integrable  holomorphic quadratic differential on $R_n.$  The statement follows from \tes\ theorem.
 \end{proof}

 We let $p$ be a point on the core curve $\gamma$ of ${\mathcal C}_{\gamma}$ corresponding to $\zeta = 0$ in the strip.  Similarly, we let 
 $p_n$ corresponding to the same point $\zeta=0$ in the strip but in the factor space 
 $$ {\mathcal C}_{\gamma_n} \cong Strip_{\gamma}/(\zeta \mapsto \zeta+2\pi n).$$

 With these choices  Bers' fiber space theorem  identifies the fiber spaces $F(Teich(R_n))$ and $F(Teich(R))$ with \te\ spaces $Teich(R_n-p_n)$ and $Teich(R-p),$ that is, the \te\ spaces of the punctured Riemann surfaces $R-p$ and $R_n-p_n.$
 
 \begin{lemma}\label{lemmaseven} Assume $q_n$ is the holomorphic quadratic differential on $R$ whose restriction to the maximal annulus ${\mathcal A}_{r^n}$ is equal to $(d\zeta)^2.$ Also assume $V_n$ is a tangent vector to $Teich(R_n)$ for which 
 $\overline{\partial} V_n = |q_n|/q_n$ and $R_n$ is  Riemann surface with marked complex structure $\tau_n.$ If $C_{Teich(R)}(\tau,V)=K_{Teich(R)}(\tau,V)$ then for every integer $n \geq 1$
 $$C_{F(Teich(R_n))}(\tau_n,V_n)=K_{F(Teich(R_n))}(\tau_n,V_n).$$
 \end{lemma}
 \begin{proof} Choose the point $p_n$ on $R_n$ which corresponds in the $\zeta$ coordinate to $\zeta=0.$   
 The integrable holomorphic quadratic differential $q_n$ on $R_n$ is also integrable and holomorphic on $R_n-p_n$ and the \te\ disc $[t|q_n|/q_n]$ in $Teich(R_n)$ lifts to the \te\ disc
 $[[t|q_n|/q_n]]$ in $Teich(R_n-p_n).$  Thus the result follows from Lemma \ref{lemma2}.
 \end{proof}

\begin{lemma}\label{hyperbolicity} Assume the hypotheses of the previous lemma and that $Teich(R)$ has dimension more than $1.$  Also assume $\tilde{p}$ is the point $0$ on the horizontal line $\eta=0$ in the strip and $p$ is equal to the point the core curve of the annulus ${\mathcal A}_r$ covered by $\tilde{p}.$  Also put $p_n$ equal to the point on the annulus ${\mathcal A}_{r^n}$ covered by $\tilde{p}.$   Then $Spin_{2 \pi n}$ represents 
an element of the mapping class group of $Teich(R_n-p)$ that preserves the \te\ disc 
$${\mathbb D}(q_n)=\{[[s|q_n|/q_n]] \in Teich(R_n-p): |s|<1\}$$  and the restriction of  $Spin_{2\pi n}$  to this disc is a hyperbolic transformation.
\end{lemma}
 \begin{proof}
 $Spin_{2 \pi n}$ preserves ${\mathbb D}(q_n)$ because on $R_n$ it is homotopic to the translation $\zeta \mapsto \zeta+ 2 \pi n$ and $q_n(\zeta) (d \zeta)^2= (d \zeta)^2$ is automorphic for this translation.  Since the dimension of $Teich(R)$ is more than $1$ the translation length of this transformation is positive.
 
 Note that  ${\mathbb D}(q_n)$ is a conformal disc lying in the Teichmueller space $Teich(R_n).$  On it the mapping class element $Spin_{2 \pi n}$ acts as the translation $\zeta \mapsto \zeta + 2 \pi n.$   Because of the extremal length estimates given, the two points $+\infty$ and $-\infty$ on the boundary of the strip are realized by distinct points on the extremal length boundary of Teichmuller space, which implies that $Spin_{2 \pi n}$ realizes a hyperbolic transformation on ${\mathbb D}(q_n).$
 \end{proof}

 Several authors have used spin elements of the mapping class group in different contexts and different ways.  But none of them in the manner used here where $Spin_{2 \pi n}$ is viewed as spinning once around a simple curve on a changing family surfaces $R_n$ whose complexity increases proportionately to $n.$   See \cite{Birmanbook} and \cite{Kra8}. 
 
 In the next section we find a contradiction by using the same argument used to prove Theorem \ref{thm5} in section 6 together with the estimates described in the next section.  

 \end{section}

\begin{section}{The waist sequence}\label{lengthofhs}

 Just as in previous sections we fix the point $p$ in a Riemann surface  $R$ lying on the 
 core curve of a separating cylindrical differential $q_{\gamma}$ and assume $p$ 
 corresponds to $\zeta=0$ in the maximal cylindrical strip ${\mathcal C}_{\gamma}$ 
 which is covered by the  horizontal strip $Strip(q_{\gamma}).$    \begin{definition}\label{dfnalpha} Let   
 ${\mathcal C}_n$ be the cylinder conformal to 
 $$Strip(q_n)/(\zeta \mapsto \zeta + 2 \pi n)$$
 and $[[\mu_t]]$ the equivalence class of Beltrami coefficient of $Spin_t$ in 
 $F(Teich(R_n)).$
 Then we call $$\alpha_n(t)=[[\mu_t]] \in Teich(R_n-p)$$ 
  for $\mu_t, 0 \leq t \leq 2 \pi n$ the $n$-th spin curve.  
  \end{definition}

 The spin curve $\alpha_n(t)$ determines a waist curve $\beta_n(t)$ in ${\mathbb C}.$   To obtain $\beta_n$ we  use the Bers isomorphism that realizes 
 $Teich(R_n-p)$ as a fiber space  over $Teich(R_n).$  In the normalization used by Bers $R_n$ is covered by Fuchsian group acting on the upper half plane and one assumes the point $p$ on the surface $R_n$ is covered by the point $i$ in the upper half plane.  Then Bers isomorphism of $Teich(R_n-p_n)$ with $F(Teich(R_n))$  is realized  by the map 
  \begin{equation}\label{secondentry}Teich(R_n-p_n) \ni [[\mu]] \mapsto \left([\mu], w^{[\mu]}(i)\right) \in Teich(R_n) \times \Delta_{\mu}. 
 \end{equation}
 Here $\Delta_{\mu}$ is by definition the unit disc $\Delta$ with moving complex structures given by the Beltrami coefficients $\mu$ that represent elements $[\mu]$ of $Teich(R_n).$

 \begin{definition}{\bf [the waist sequence]}\label{hs} The $n$-th  waist curve is the plane curve 
 $$\beta_n(t) = w^{[\mu_t]}(i)$$
 given by the second entry in (\ref{secondentry}), where $[\mu_t] = \Psi(\alpha_n(t))$ and 
 $\alpha_n(t)=[[\mu_t]] \in Teich(R_n-p),$ $0 \leq t \leq 2 \pi n.$
 \end{definition}
   We will use the upper and lower bounds from   
  Theorem \ref{thm2}  to estimate the length of $\beta_n(t), 0 \leq t \leq 2 \pi n$ with respect to the \te\ density $\lambda_{{\mathcal A}_{r^n}}$ on the annulus ${\mathcal A}_{r^n}$ in two ways. 
  From Theorem \ref{thm2} of section 3 we have for all Riemann surfaces $R$
  \begin{equation}\label{first}
  (1/2) \rho_R \leq \lambda_R \leq \rho_R.
   \end{equation}
      For each positive integer $n$ we apply this inequality to the subannuli ${\mathcal A}_{r^n}$ contained in $R_n.$

We use notation $\lambda_n(\beta_n),$ $\rho_n(\beta_n)$ and $c_n(\beta_n)$ to denote the lengths of the curve $\beta_n,$ respectively, with respect to the 
\te\ density $\lambda_n$, the Poincar\'e density $\rho_n$ and the Carath\'eodory
density $c_n$ on ${\mathcal A}_{r^n},$ respectively. 
    
  Using Lemma \ref{lemma2} and the forgetful map $F(Teich(R_n)) \rightarrow Teich(R_n),$ we see that our assumption 
  $$C_{Teich(R)} = K_{Teich(R)}$$  implies the existence of  a holomorphic function $g_n = 
  \widetilde{Ev} \circ \Psi$ defined on $Teich({\mathcal A}_{r^n})$ with image in the $\zeta$-strip.  Here $\Psi$ is the forgetful map in Figure 3 with $r$ replaced by $r^n$ and 
  $\widetilde{Ev}$ is the lift to the unit disc of the evaluation map $Ev$ that takes $[[g_n]]$ to $g_n(p)$ in the unit disc.  Its image is automorphic with respect to $\zeta \mapsto \zeta + 2\pi n.$  Thus it induces a  the annulus $A_{r^n}.$     
  Identifying this strip by a conformal map with the unit disc $\Delta$ we can view $g_n$ as mapping 
  onto $\Delta$ in such a way that  that $\widetilde{Ev}$ maps to the unit disc and takes the point $1$ 
  in  ${\mathcal A}_{r^n}$ to the point $0$ in $\Delta.$  Then  by Lemma \ref{lemma2} we obtain
  \begin{equation}\label{hyp5}
  g_n \circ f_n(t)=t {\rm \ for \ all \ t \in \Delta.}
  \end{equation}  
   The 
  isometric inclusion of the  \te\ disc 
  $${\mathbb D}_{q_n}= \{[t|q_n|/q_n]\} \subset Teich(\mathcal{A}_{r^n})$$ implies that  
  $\rho_n(\beta_n)=c_n(\beta_n).$
  Thus  we have 
  \begin{equation}\label{upperbound2}
  \lambda_n(\beta_n) \leq \rho_n(\beta_n) = c_n(\beta_n) \leq (M/n) \cdot n = M.
  \end{equation}
  The first inequality in the middle of (\ref{upperbound2}) follows from the right hand inequality (\ref{hyp5})
  and the existence of the number $M$ independent of $n$  in the 
  second inequality of  (\ref{upperbound2}) follows from the left hand inequality of  (\ref{300})  in Theorem \ref{thmsix}.

 Using the left hand side of (\ref{first}) to estimate the lengths $\lambda_{{\mathcal A}_{r^n}}(\beta_n)$   from below we have have
   \begin{equation}\label{above}\lambda_{{\mathcal A}_{r^n}}(\beta_n) \geq (1/2) \rho_{{\mathcal A}_{r^n}}(\beta_n) \geq \frac{n}{2\log r}.\end{equation}
    This estimate implies $\lambda_{{\mathcal A}_{r^n}}(\beta_n)$ is unbounded in its dependence on $n$ which contradicts (\ref{upperbound2}) and proves Theorem \ref{maintheorem}.
  \end{section}

\begin{section}*{Appendix A: Schwarz-Pick metrics\ \  \ \  \ \ \ \ \ \ \ \ \ \ \ \ \ \ \ \ \  \  \ \  \ \ \ \ \ \ \ \ \ \ \ \ \ \ \ \ \ }

  If $h$ is a holomorphic function from a complex manifold $M$ to a complex manifold $N,$  
  the derivative $dh_{\tau}$ of $h$ at $\tau \in M$  is a complex linear map from the tangent 
  space to $M$ at $\tau$ to the tangent space to $N$ at $h(\tau).$  From Schwarz's lemma it follows both $C_M$ and $K_M$ have the pull-back contracting property, 
  namely, 
  for any point $\tau \in M$ and tangent vector $V$ at $\tau,$  $C$ and $K$ satisfy
           \begin{equation}\label{pullbackK}K_N(h(\tau),dh_{\tau}(V)) \leq K_M(\tau,V)\end{equation}
           and
           \begin{equation}\label{pullbackC}
           C_N(h(\tau),dh_{\tau}(V)) \leq C_M(\tau,V).
           \end{equation}
Any infinitesimal form on a complex manifold  satisfying this pull-back contracting property for all holomorphic functions $h$ is called a  Schwarz-Pick metric by Harris in \cite{Harr79} and by Harris and Earle in \cite{EarleHarris}.  They also observe that any infinitesimal form with this property must necessarily lie between $C_M(\tau,V)$ and $K_M(\tau,V).$
\end{section}

\begin{section}*{Appendix B: Equivalence of the infinitesial forms of $C$ and $K.$}

 The Bers embedding is a biholomorphic map 
 $\Phi_{\tau}$ from $Teich(R)$ onto a bounded simply connected domain  in the Banach space of bounded cusp forms on the Riemann surface $R_{\tau}$ which is the differentiable surface together with a marked  complex structure $\tau$ indicated by the point $\tau \in Teich(R).$  A function 
 $\varphi(z)$ holomorphic in the upper half plane is a bounded cusp form for $R_{\tau}$ if it satisfies the inequality 
 $$ \sup_{\{z=x+iy: y>0\}} |\varphi(z)|  y^2< \infty$$
 and if $$\varphi(\gamma(z))\gamma'(z)^2 = \varphi(z)$$
 for every $\gamma$ in a Fuchsian covering group $\Gamma_{\tau}$ of $R_{\tau}.$
 
 We consider the set of quasiconformal self mappings  $z \mapsto w$ of the complex plane for which 
 $w \circ \Gamma_{\tau} \circ w^{-1}$ is a subgroup of the group of Moebius self-mappings of $\overline{\mathbb C}.$   Any $w$ in this set satisfies
 $\{w,z\} = \varphi(z)$ is a bounded cusp form in the lower half plane where
 $\{w,z\}$ is notation for the Schwarzian derivative, namely, 
 $$\{w,z\} = N' - (1/2) N^2 {\rm \ and \ } N= \frac{w''}{w'}.$$

 Given any Beltrami coefficient $\mu$ for which $||\mu(z)||_{\infty} < 1$ and for which
 $$\mu(\gamma(z)| \frac{\overline{\gamma'(z)}}{\gamma'(z)} = \mu(z)$$ the quasiconformal selfmapping $z \mapsto w$ of ${\mathbb C}$ with Beltrami coefficient $\mu$ in the upper half plane and identically equal to $0$ in the lower half plane has the property that  holomorphic 
 mapping $z \mapsto w$ is holomorphic and univalent in the lower half plane.
 We put $\mu$ equal to the Beltrami coefficient of $w$ and $\Phi_{\tau}(\mu)=\{w,z\}.$

 Then Bers shows that the image of 
 $\Phi_{\tau}: Teich(R_{\tau}) \rightarrow B_{\tau}$
  in $B_{\tau}$ contains the open ball of radius $2$  and is contained in the ball of radius $6,$ 
  \cite{Ahlforsbook5}.  
    Assume $V$ is a tangent vector $||V||_{B_{\tau}}=1.$  Then the complex linear map $V \mapsto (1/6)V$ extends by the Hahn-Banach theorem to a complex linear map $L: B_{\tau} \rightarrow {\mathbb C}$  
    with 
  $||L||\leq 1/6.$  Then because $\Phi_{\tau}(Teich(R))$ 
  is contained in the ball of radius $6$,
   $L(\Phi_{\tau}(Teich(R))$  is contained in the unit disc and $L  \in {\mathcal G}.$ Thus from definition (\ref{C111})  
  $$C_{Teich(R)}(\tau,V) \geq 1/6.$$
  On the other hand, put $f(t)=2tV.$ Then since the image of $\Phi_{\tau}$ contains 
  $f(\Delta)$ from definition (\ref{K}) 
  $$K_{Teich(R)}(\tau,V) \leq 1/2.$$ Putting the preceding two inequalities together we find
 \begin{equation}\label{KC}
 (1/3) K_{Teich(R)}(\tau,V) \leq C_{Teich(R)}(\tau,V).
 \end{equation} 
 (\ref{KC}) and (\ref{C<K}) together yield (\ref{equivCK}).
 Inequality (\ref{KC}) has been pointed out to me by Stergios Antonokoudis and the same argument is given  by Miyachi  in \cite{Miyachi4} to prove the parallel result for asymptotic \te\ space. 
 \end{section}

\begin{section}*{Appendix C: Kra's equality theorem\ \  \ \  \ \ \ \ \ \ \ \ \ \ \ \ \ \ \ \ \  \  \ \ \ \ \ \ \ \ \ \ \ \ \ \ \ }
  Inequality (\ref{inequality})  is a feature that distinguishes between different tangent vectors. In fact by a theorem of  Kra \cite{Kra1}  for some tangent vectors one can have equality
 \begin{equation}\label{Kra'sequation}C(\tau,V)=K(\tau,V).
 \end{equation}
 Kra's theorem states (\ref{Kra'sequation}) holds when $V$ has the form
$\overline{\partial}V=|q|/q$ where $q = \omega^2$  and $\omega$ is a multiple of an element of a canonical basis for the abelian differentials of the first kind, see Theorem 4.2 of McMullen's paper \cite{McM10}.

  Kra's result 
is obtained by using the Riemann period relations and  Ahlfors' first variation formula for the motion of a point $w^{\mu}(z)$ in terms of the variation of the Beltrami coefficient $\mu,$ see \cite{Ahlforsbook5}.  This formula is sometimes called Rauch's variation formula  \cite{Rauch} when it is applied to  the variation of a diagonal
entry $\int_{\alpha_i} \omega_i$ of the period matrix.  
When first understood Rauch's observation was important because it showed that the complex structure coming from the unit ball in the complex Banach space of Beltrami coefficients had to coincide with the complex structure coming from the embedding of \te\ space in the Siegel upper half plane.

We call a disc that has the form $[t|q|/q]$ for $|t|<1$
where $q=\omega_i^2$  where $\omega_i$ is a diagonal entry in the period matrix an abelian \te\ disc.  
A consequence of this infinitesimal result is that for any two points  ${\tau_1}$ 
and ${\tau_2}$ in the same abelian \te\ disc, $C({\tau_1},{\tau_2})=K({\tau_1},{\tau_2}).$ 
Note that the usual simple closed curves $\alpha_j$ and $\beta_j,$  $1\leq j \leq g,$ taken as a basis for  the homology group of a surface of genus $g$ are all non-separating and homologous to zero.

 An instructive example is a  genus 2 surface $X$ with is a separating cylindrical quadratic differential  $q$
that separates the surface into two genus 1 components.  It is possible for $q$ to have one double zero in each  component.  In that case $q=\omega^2$ where $\omega$ is a holomorphic differential $1$-form on $X.$  By Theorem \ref{maintheorem} it cannot be a differential occurring in the form $\int_{\alpha_i} \omega_i$
where $\omega_i$ is part of a homology basis.

In contrast to  abelian tangent vectors, any tangent vector $V$  corresponding to a separating cylindrical differential will originate from what are  sometimes called ${\mathcal F}$-structures \cite{EarleGardiner1} and sometimes called  half-translation structures  \cite{FortierBourqueRafi}.    \end{section}

\begin{section}*{Appendix D: Lakic's dual formulas for the \te\ density \ \ \ \ \ \ \ \ \ \ \ \ \ \ }
If $\varphi$ is a holomorphic $q$-differential  $\varphi$ on $R$ except for possibly a simple pole at a point
$p$ in $R$  we use the notation $res_p(\varphi)$ to mean the residue of $\varphi$ at $p.$
In general, if $\varphi$ is a $q$-differential then $res_p(\varphi)$ is a $(q-1)$-differential.  In Definition \ref{tdensity} we have defined the \te\ density at a point $p$ on any Riemann surface.

 \begin{theorem}\label{thm3} Let $p$ be a point in a Riemann surface  $R$ and $Q(R-p)$ be the 
 space of quadratic differentials $\varphi$  which are holomorphic on $R-p$ but which may have a 
 simple pole at $p$ and   for which $||\varphi|| = \int\!\!\int_{R} |\varphi| < \infty.$   Then the \te\ density 
 $\lambda(t)dt$ on $R$ satisfies
         $$\lambda_{\tau}(t)|dt| = \sup_{\varphi}
         \left\{|\pi \cdot res_p(\varphi)| \right\}
         = \inf_{V} \{||V_{\overline{z}}||_{\infty} \}, $$ 
         where the supremum is taken over integrable holomorphic quadratic differentials $\varphi$ in 
         $Q(R-p)$ with $||\varphi|| \leq 1$ and
        where the  infimum is taken over all continuous vector fields  $V$ on $R$ that are equal to zero on the boundary of $R$ and equal to $1$ at $p.$        
        \end{theorem}
 \begin{proof} See \cite{GardinerLakic1} and \cite{GardinerLakic3}.
 \end{proof}
\end{section}

\nocite{GardinerLakic1}
\nocite{GardinerLakic3}
\nocite{Ahlfors6}
\nocite{Krushkal2}
\nocite{Royden2}
\nocite{Kobayashibook}
\bibliographystyle{amsplain}

\bibliography{../Articles,../books}

\end{document}